\documentclass[a4paper]{amsart}
\usepackage{amsthm,amsfonts,amsmath, amssymb}

\renewcommand{\leq}{\leqslant}
\renewcommand{\ge}{\geqslant}
\renewcommand{\le}{\leqslant}
\usepackage[capitalise]{cleveref}

\usepackage{mathpazo}

\usepackage{geometry}                
\usepackage{anysize}
\marginsize{3cm}{3cm}{3cm}{3cm}

\renewcommand{\baselinestretch}{1.25}
\usepackage{mathtools}

\newcommand{\R}{\mathbb{R}}

\usepackage{graphicx,float,enumerate}
\usepackage[ruled,boxed,commentsnumbered,norelsize]{algorithm2e}
\usepackage{verbatim}
\usepackage{url}
\usepackage{epstopdf}
\usepackage{subfigure}

\newtheorem{theorem}{Theorem}

\newtheorem{lemma}[theorem]{Lemma}
\newtheorem{proposition}[theorem]{Proposition}

\newtheorem{question}[theorem]{Question}

\theoremstyle{definition}

\theoremstyle{remark}

\usepackage{amsaddr}
\title{Lower bound theorems for general polytopes}
\author{Guillermo Pineda-Villavicencio}
\address{Centre for Informatics and Applied Optimisation, Federation University,  Ballarat, Victoria 3353, Australia}
\email{\texttt{work@guillermo.com.au}}
\author{Julien Ugon}
\address{School of Information Technology, Deakin University,  Burwood, Victoria 3125, Australia}
\email{\texttt{julien.ugon@deakin.edu.au}}
\author{David Yost}
\address{Centre for Informatics and Applied Optimisation, Federation University, Ballarat, Victoria 3353, Australia}
\email{\texttt{d.yost@federation.edu.au}}
\thanks{\copyright 2019. This version is made available under CC-BY-NC-ND 4.0, \url{creativecommons.org/licenses/by-nc-nd/4.0/}}

\def\h{\frac{1}{2}}
\begin{document}

\begin{abstract}
For a $d$-dimensional polytope with $v$ vertices,  $d+1\le v\le2d$, we calculate precisely the minimum possible number of $m$-dimensional faces, when $m=1$ or
$m\ge0.62d$. This confirms a conjecture of Gr\"unbaum, for these values of $m$. For $v=2d+1$, we solve the same problem when $m=1$ or $d-2$; the solution was
already known for $m= d-1$. In all these cases, we give a characterisation of  the minimising polytopes. We also show that there are many gaps in the
possible number of $m$-faces: for example, there is no polytope with 80 edges in dimension 10, and a polytope with 407 edges  can have dimension at most 23.
\end{abstract}\maketitle

\section{Introduction}
A problem which has long been of interest is determining the possible number of $m$-dimensional faces of a polytope, given the number of vertices; see, for instance, \cite[pp.~1152-1153]{Gru70} or \cite[Sec.~10.2]{G}. Most of this paper is concerned with the case $m=1$. Accordingly, we consider the set $E(v,d)=\{e:$ there is a $d$-polytope with $v$ vertices and $e$ edges$\},$
and define, following Gr\"unbaum's notation \cite[p 184]{G},
$$\phi(v,d)={d+1\choose2}+{d\choose2}-{2d+1-v\choose2}.$$

 It is convenient to note two equivalent expressions for $\phi$,
$$\phi(v,d)={v\choose2}-2{v-d\choose2}$$
and
$$\phi(d+k,d)=\h d(d+k)+\h(k-1)(d-k).$$

Our  fundamental result, in \S3, is  to prove Gr\"unbaum's conjecture \cite[p.~183]{G} that $\phi(v,d)=\min E(v,d)$ for $d<v\le2d$. Gr\"unbaum proved this for $d<v\le d+4$.

In \S4, we  prove that  $\min E(2d+1,d)=d^2+d-1$ for every  $d\ne4$. This was well known in the cases $d=2$ or 3, and Gr\"unbaum \cite[p 193]{G} noted that $\min E(9,4)=18$. In all cases, we also characterize, up to combinatorial equivalence, the $d$-polytopes with minimal number of edges.

In fact, the function $\phi$ was called $\phi_1$ in \cite{G}; for simplicity we will {(except in the final section)} continue to drop the subscript $_1$. Gr\"unbaum actually defined a function
$$\phi_m(v,d)={d+1\choose m+1}+{d\choose m+1}-{2d+1-v\choose m+1}$$
for each $m\le d$, and conjectured that this is  the minimum possible value of $f_m(P)$ over all $d$-polytopes  $P$ with $v$ vertices, provided $v\le2d$.
Here $f_m(P)$ denotes as usual the number of $m$-dimensional faces of  $P$. We will say  more about higher dimensional faces in \S6.

Precise upper bounds for the numbers of edges are easy to obtain. A well known result of Steinitz \cite[Sec.~10.3]{G} asserts that $\max E(v,3)=3v-6$, and the existence of cyclic polytopes shows that
$\max E(v,d)={v\choose2}$ for $d\ge4$. Since cyclic polytopes are simplicial, the upper bound question for general polytopes has the same solution as the  upper bound question for simplicial polytopes.

So we concentrate on lower bounds, which  for general polytopes have been elusive to obtain.  The most important result to date is
Barnette's Lower Bound Theorem for  simplicial polytopes.

\begin{theorem}\label{LBT} \cite{Ba} For  any simplicial $d$-polytope with $v$ vertices and $e$ edges, we have
$$e\ge dv-{d+1\choose2}.$$

\end{theorem}

Barnette also showed that there exist simplicial $d$-polytopes, namely the
stacked polytopes, with precisely this many edges. Kalai's Rigidity Theorem \cite{K} asserts that this  lower bound is still correct under the weaker assumption that every 2-face is a triangle. However little seems to be known for general polytopes.

It is well known that there is no 3-polytope with 7 edges. In \S5, we show that there are many more gaps in the possible number of edges of higher dimensional polytopes.

\section{ Polytopes with many vertices}

This  section justifies our focus on low values of $v$. However it can be skipped, as the rest of the paper does not essentially depend on it.

Naturally every vertex in a $d$-polytope has degree at least $d$; a vertex with degree exactly $d$ is called {\it simple}. A $d$-polytope is {\it simple} if every vertex is simple, which is equivalent to saying that $2e=dv$. In particular, the existence of a simple $d$-polytope with $v$ vertices implies that $\min E(v,d)=\h dv$. Conversely, if $\min E(v,d)=\h dv$, this minimum must be achieved by simple $d$-polytope with $v$ vertices.

 We show here that the problem of calculating $\min E(v,d)$ is more interesting for small
values of $v$. This is not new, and our estimate for $K$ is not the best possible, but our argument is completely elementary.

\begin{proposition}
For each $d$, there is an integer $K$ such that, for all $v > K$, if either $v$ or $d$ is even,
then $\min E(v,d)=\h dv$.
\end{proposition}

\begin{proof} First note that for a fixed integer $k$, $\{ak+b(k+1): a,b\ge0,\, a, b\in \mathbb{N}\}$ contains the interval $[k^2-k,\infty)$. A special case of this is that  $\{a(d-2)+b(d-1): a,b\ge0\}$ contains  the interval $[(d-2)(d-3),\infty)$. Multiplying everything by 2 we conclude that
$\{a(2d-4)+2b(d-1): a,b\ge0\}$ contains every even number from $2(d-2)(d-3)$  onwards.
Now cutting a vertex from a simple $d$-polytope gives us another simple $d$-polytope with $d-1$ more vertices;
while cutting an edge from a simple $d$-polytope gives us another simple polytope with $2d-4$ more vertices.

If $d$ is even, it follows that for every  odd $v\ge d+1 + 2(d-2)(d-3)$ and  for every even $v\ge2d+ 2(d-2)(d-3)$ there exists a simple $d$-polytope with $v$ vertices.

 When  $d$ is odd, we see that for every  even $v\ge  d+1 + 2(d-2)(d-3)$ there exists a simple $d$-polytope with $v$ vertices.
\end{proof}

For the case when $d$ and $v$ are both odd, we can prove that $\min E(v,d)=\h d(v+1)-1$ for all sufficiently large $v$. The proof of this lies somewhat deeper, and details will appear elsewhere \cite{PUY}.

This proof gives an estimate for $K(d)$ of about $2d^2$. The original proof of Lee \cite[Corollary 4.4.15]{L} gave a weaker estimate, about $d^3$, more precisely a polynomial with leading term $d^3/24$. (The lower order terms were different, depending on the parity of $d$, but their coefficients were all positive.) This was improved by Bj{\"o}rner and Linusson \cite{BL}, motivated by work of Prabhu \cite{P}, to $\sqrt{2d^3}$ when $d$ is even and $\sqrt{d^3}$ when $d$ is odd. The proofs of Lee and of Bj{\"o}rner \& Linusson both depended on the $g$-theorem \cite[\S8.6]{Z}. Prabhu's did not, but it is still less elementary than ours.

\section{Polytopes with up to $2d$ vertices}

It is often easier to work with the {\it excess degree } of a $d$-polytope, which we define as

$$\xi(P)=2e-dv=\sum_{v\in V}(\deg v -d).$$
A polytope is simple if, and only if, its excess degree is 0. Note that for fixed $d$ and $v$, the possible values of the excess are either all even or all odd.
The excess degree function is studied in further detail in \cite{PUY}.

Throughout, we will use the word {\bf prism} to mean a prism whose base is a simplex. Such a prism is also called a {\it simplicial prism}.
An object of natural interest to us is
 any multifold pyramid over a prism based on a simplex; this is the same as the free join of a prism and a simplex. Any such polytope has three disjoint simplex faces (not necessarily unique), whose union contains every vertex. Accordingly, we will  define a  ``triplex" as any such free join. Clearly a $d$-dimensional triplex has at most $2d$ vertices.

To be more precise, we introduce the notation $M_{k,d-k}$ for any $(d-k)$-fold pyramid over a $k$-dimensional simplicial prism, $1\le k\le d$. Any triplex is of this form for some values of $d$ and $k$. Clearly $M_{1,d-1}$ is a simplex, and $M_{d,0}$ is a prism. Each simplex, and each triplex $M_{2,d-2}$, is a multiplex as defined by Bisztriczky \cite{B}, but other triplices are not.

It is worth noting that for $k\ge3$, $M_{k,d-k}$ has three types of facet:
\begin{enumerate}
\item
{$d-k$ facets of the form} $M_{k,d-k-1}$  (by definition),
\item
{$k$ facets of the form} $M_{k-1,d-k}$ (each of which has two vertices outside),
\item
{and 2 facets of the form} $M_{1,d-2}$ (both simplices).
\end{enumerate}

For $k=2$, the latter two forms coincide.
{If $k\ge2$, then} $M_{k,d-k}$ has $d+2$ facets altogether. More generally, let us note here that  if $P=M_{k,d-k}$ is a triplex with $d+k$ vertices, then
$$f_m(M_{k,d-k})=\phi_m(d+k,d).$$

In general, if $P$ is a pyramid with base $F$, then $f_m(P)=f_m(F)+f_{m-1}(F)$, so this calculation is quite routine.
We will show in this section that $M_{k,d-k}$ is (up to combinatorial equivalence) the unique $d$-polytope which minimises the number of edges of a $d$-polytope with $d+k$ vertices (for $1\le k\le d$).
 We will show in the last section that $M_{k,d-k}$ is also the unique $d$-polytope which minimises the number of $m$-faces of a $d$-polytope with $d+k$ vertices, at least for $0.62d\le m\le d-2$.

\begin{figure}[h]
\centering
\includegraphics[scale=.9]{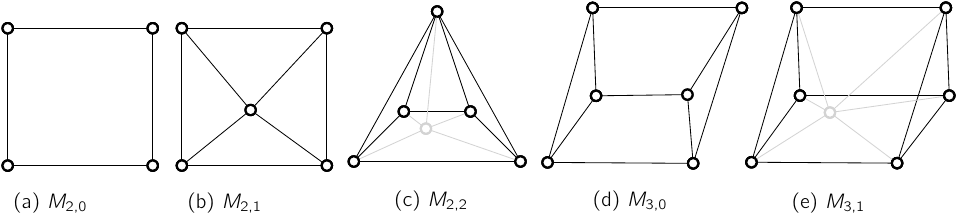}
\label{fig:multiplexes}
\caption{Schlegel diagrams of triplices.}
\end{figure}

The following identity is  useful for us to know.

\begin{lemma}\label{boring}
$$\phi(d+k-n,d-1)+nd-{n \choose 2}=\phi(d+k,d)+(k-n)(n-2).$$
\end{lemma}

Often, we will want to estimate the number of edges {in a polytope $P$ which involves a set $S$ of vertices} lying outside a given facet $F$. The following result gives this estimate, but is more general.

\begin{lemma}\label{outside} Let $S$ be a set of  $n$   vertices of  a $d$-polytope $P$, with $n\le d$.
Then the total number of edges containing at least one vertex in $S$ is at least
$nd-{n \choose 2}$.
\end{lemma}

\begin{proof}

Each vertex in $S$ has degree at least $d$, and at most $n \choose 2$ edges connect them to one another.
Thus the total number of such edges  is at least
$n(d-(n-1))+{n \choose 2}$.
\end{proof}

A polytope is said to be {\it decomposable} if it can be expressed as the sum (or Minkowski sum) of two dissimilar  polytopes; this concept also makes sense for general convex bodies.  (Recall that the Minkowski sum $A+B$ of two convex bodies $A$ and $B$ is simply $\{a+b:a\in A \hbox{ and } b\in B\}$, and that two polytopes are {\it similar} if one can be obtained from the other by a dilation and a translation.)

Many results about decomposability, in particular the next theorem, are combinatorial rather than geometric in nature, and thus apply to all members of a combinatorial class. Recall that two polytopes are called combinatorially equivalent if they have isomorphic face lattices. Many properties of polytopes are preserved by combinatorial equivalence, but not all. In particular, decomposability is not preserved by combinatorial equivalence. That is, a class of combinatorially equivalent polytopes can have one member which is decomposable and another member which is indecomposable. This issue does not arise in this paper; indeed a consequence of our current work is that a combinatorial class of $d$-polytopes can suffer this ambiguity only if its members have at least $2d+2$ vertices. See \cite{PYdecomp} and the references therein for further discussion of this topic.

The following  sufficient condition  will be useful to us several times. It is due to Shephard; for another proof, see \cite[Prop. 5]{PYmore}.  Let us say that a facet $F$ of a polytope $P$ has {\it Shephard's property} if for every vertex $v \in F$, there exists exactly one edge in $P$ that is incident to $v$ and does not lie in $F$.

\begin{theorem}
[{\cite[Result (15)]{S}}] \label{shp} If a polytope $P$ has a facet $F$ with   Shephard's property, and there are at least two vertices outside $F$, then $P$ is
decomposable. In particular, any simple polytope other than a simplex is decomposable.
\end{theorem}

It is easily verified that the prism has  $2d$ vertices and $d^2$ edges, and is simple and decomposable.
Several times, we will need to know that the converse  is true. This was proved in \cite[Theorem 7.1, page 39]{Kal} but never published; a different proof is given in \cite[Theorem 10]{PYmore}.

\begin{proposition}\label{decomp} Let $P$ be a $d$-polytope with $2d$ or fewer vertices. Then the following are equivalent.

(i) $P$ is decomposable,

(ii) $P$ is simple but not a simplex,

(iii) $P$ is a simplicial prism,

(iv) $P$ has exactly $2d$ vertices and $d^2$ edges.
\end{proposition}

The next result not only verifies Gr\"unbaum's conjecture, but also establishes uniqueness of the minimizing polytope.

\begin{theorem}\label{big} Let $P$ be a $d$-dimensional polytope with $d+k$ vertices, where $0<k\leq d$. Then $P$ has at least $\phi(d+k,d)$ edges,  with equality if and only if
 $P$ is the $(d-k)$-fold pyramid over the $k$-dimensional prism.

\end{theorem}

\begin{proof}  It has already been noted   that the triplex $M_{k,d-k}$ has precisely $d+k$ vertices and  $\phi(d+k,d)$ edges. Since $\phi(2d,d)=d^2$,  the previous lemma establishes the case $k=d$.

We will proceed by induction on $k$; and for fixed $k$ we proceed by induction on $d$. The case $k=1$ is trivial and $k=2$ is both easy and well known \cite[Sec.~6.1,10.2]{G}.

Now we fix $k>2$ and proceed by induction on $d$.

Let $F$ be any facet of $P$, and let $n$ be the number of vertices not in $F$. Then $F$ has $d+k-n$ vertices and $0<n\leq k$.

First suppose $n=1$. Then $P$ is a pyramid over $F$,
and $f_1(P)=f_1(F)+f_0(F)$.
If $F$ is a triplex, then so is $P$, and we are finished. Otherwise, by induction on $d$,
$F$ has strictly more than $\phi(d+k-1,d-1)$ edges, and $P$ must have strictly more than $\phi(d+k-1,d-1)+(d+k-1) = \phi(d+k,d)$  edges.

For $n>1$, we can only estimate the number of edges outside $F$. By Lemma 2, this is at least $nd-{n \choose 2}$.

 \cref{boring} above establishes the conclusion if either $F$ is not a triplex, or $2< n< k$.

Consider the case when $n=2$, $k>n$, and $F$ is a triplex, and call $u,v$  the two vertices outside $F$. Then, since $F$ has $d-1+k-1$ vertices,  it must be $M_{k-1,d-k}$. Since $d>k$, $F$ is a pyramid over some ridge $R$ with $d-2+k-1$ vertices, i.e. $M_{k-1,d-k-1}$. The case when $P$ is a pyramid has been dealt with, so the other facet, $G$ say, containing the ridge $R$ must be a pyramid, say with apex $u$. Consider separately the edges in $F$, the edges joining $u$ to $F$, and the edges containing $v$:  then the total number of edges in the polytope is at least
$$\phi(d-1+k-1,d-1)+(d-2+k-1)+d=\phi(d+k,d)+k-2
$$
which clearly exceeds $\phi(d+k,d)$.

 The only remaining case is that $k=n$ for every facet. Then $P$ is simplicial, and we can apply the Lower Bound Theorem.
\end{proof}

The previous result fits neatly into a result about the excess degree.

\begin{theorem}\label{excess} Let $P$ be a $d$-dimensional polytope with $d+k$ vertices,  with $k>0$.

(i) If $d\ge4$, then $\xi(P)\le(k-1)(d+k)$, with equality in the case of cyclic polytopes.  For $d=3$, we have $\xi(P)\le3(k-1)$, with equality precisely for simplicial polyhedra.

(ii) If every 2-face of $P$ is a triangle, in particular if $P$ is simplicial, then $\xi(P)\ge(k-1)d$.

(iii) If $k\le d$, then $\xi(P)\ge(k-1)(d-k)$, with equality if and only if $P$ is a triplex $M_{k,d-k}$.
\end{theorem}

\begin{proof}

(i) For $d\ge4$, this is simply rewriting the obvious assertion that $e$, the number of edges, cannot exceed $v\choose 2$.

(ii) The conclusion is a rewriting of the assertion that $e\ge dv-{d+1\choose2}$. For simplicial polytopes, this is  of course  \cref{LBT}. Kalai \cite[Theorem 1.4]{K} later proved that the same conclusion holds under the weaker assumption.

(iii) Likewise, this just reformulates the previous theorem in terms of the excess degree.\end{proof}

The preceding  theorem allows us to extend known results about gaps in the possible number of edges. The case $n=1$ in the next result is very well known. The cases $n=2$ and $n=3$ are due to Gr\"unbaum \cite[p188]{G}. Our argument for $n\ge4$ follows the same pattern.

\begin{proposition}
In dimension $d=n^2+j$, where $j\ge2$, there is no $d$-polytope  whose number of edges $f_1$ satisfies ${d+n\choose2}+1<f_1<{d+n\choose2}+j-1$.

\end{proposition}

\begin{proof}
We  will use the easily established identity

$$\phi(d+n+1,d)={d+n\choose2}+d-n^2.$$

Let $P$ be a $d$-polytope with $v$ vertices and $e$ edges. If $v\le d+n$, clearly $e\le{d+n\choose2}$. If $2d\ge v\ge d+n+1$, then $e\ge\phi(v,d)\ge\phi(d+n+1,d)={d+n\choose2}+j$. If $v>2d$, then $e\ge\h dv> d^2=\phi(2d,d)>\phi(d+n+1,d)$.
\end{proof}

We used here the fact that, for fixed $d$, the quadratic function $\phi(v,d)$ is strictly increasing on the range $v\le 2d$. We do not know whether $\min E(v,d)$ is a monotonic function of $v$ (for fixed $d$). We can prove it is not strictly monotonic, as $\min E(14,6)=\min E(15,6)$; see the remarks at the end of the next section.

Gr\"unbaum was clearly aware that for $v>2d$, $\phi(v,d)$ cannot be equal to $\min E(v,d)$. Indeed it is a decreasing function of $v$ in this range. We settle the case of $2d+1$ vertices next.

\section{ Polytopes with $2d+1$ vertices}

We will define the {\it pentasm} in dimension $d$ as the Minkowski sum of a simplex and a line segment which is parallel to one 2-face,
but not parallel to any edge, of the simplex; or any polytope combinatorially equivalent to it. The same polytope is obtained if we truncate a simple vertex of the triplex $M_{2,d-2}$. In one concrete realisation, it is the convex hull of 0, $e_i$ for $1\le i\le d$ and $e_1+e_2+e_i$ for $1\le i\le d$, where $e_i$ are the standard unit vectors in $\R^d$.

\begin{figure}[h]
\centering
\includegraphics{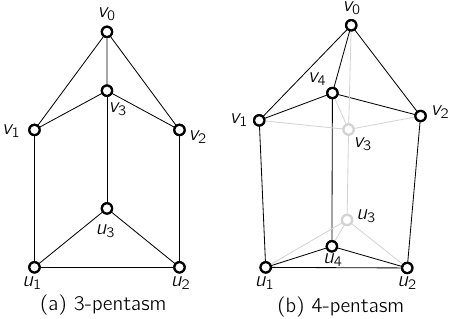}
\caption{Pentasms. (a) A 3-pentasm. (b) A 4-pentasm.}
\label{fig:pentasms}
\end{figure}

The pentasm has $2d+1$ vertices and $d^2+d-1$ edges.
 We will prove in  \cref{pentam} that $d^2+d-1$  is the minimum number of edges of a $d$-dimensional polytope with $2d+1$ vertices, for all $d$ {\it except} 4.
In particular, there is no 5-polytope with 11 vertices and 28 edges.

In general, we can label the vertices of any pentasm as $u_1,\ldots,u_d$, $v_0,v_1,\ldots,v_d$, in such a way that the edges are $[u_i,v_i]$ for $1\le i\le d$, $[u_i,u_j]$ for $1\le i<j\le d$ and $[v_i,v_j]$ for $0\le i<j\le d$ except when $(i,j)=(1,2)$.
The $d$-dimensional pentasm has  precisely $d+3$ facets:
\begin{enumerate}
\item $d-2$ pentasms of lower dimension (for each $i=3,4,\ldots,d$, the face generated by all vertices except $u_i$ and $v_i$ is a pentasm facet),

\item two prisms (one generated by all vertices except $u_1,v_1,v_0$  and the other generated by all vertices except $u_2,v_2,v_0$),

\item and three simplices (one generated by all $u_i$, another generated by all $v_i$ except $v_1$, and the third generated by all $v_i$ except $v_2$.
\end{enumerate}

The facet-vertex graph, and hence the entire face lattice, is then not hard to describe. Two of the simplices intersect in a ridge, while third is disjoint from both. Each of the first two simplices intersects one prism in a ridge, and the other in a face of dimension $d-3$. Every other pair of distinct facets intersects in a ridge. See \cref{fig:pentasms}.

Another way to view the pentasm is as the convex hull of two disjoint faces: a pentagon (with vertices $u_1,v_1,v_0,v_2,u_2$), and a $(d-2)$-dimensional prism. From this, we can verify that its $m$-dimensional faces comprise

$${d\choose m+1}+{d+1\choose m+1}-{d-1\choose m-1}\qquad{\text{simplices}},$$

$${d-2\choose m-2}\qquad{\text{pentasms, and}}$$

$${d-2\choose m}+2{d-2\choose m-1}={d-1\choose m}+{d-2\choose m-1}\qquad{\text{prisms}}.$$

Adding these up, we conclude
$$f_m(P)={d+1\choose m+1}+{d\choose m+1}+{d-1\choose m}$$
for a $d$-pentasm $P$ and $m\ge1$.

\begin{lemma}\label{pentasm}

Let $P$ be a $d$-polytope with $2d+1$ vertices, and $F$ a facet of $P$ which is a pentasm. Suppose that every vertex in $F$ belongs to only one edge not in $F$. Then $P$ is also a pentasm.
\end{lemma}

\begin{proof}
This is easy to see if $d\le3$, so assume $d\ge4$.

Denote by $x$ and $y$ the two vertices of $P$ outside $F$. Let $G$ be any other facet of $P$; we claim that $G$ must intersect $F$ in a ridge.  Since at most two vertices of $G$ lie outside $G\cap F$, the dimension of $G\cap F$ must be at least $d-3$. If its dimension were exactly $d-3$, then $G$ would be  2-fold pyramid over this subridge, with apices $x$ and $y$. But then every vertex in $G\cap F$ would be adjacent to both $x$ and $y$, contrary to hypothesis.

Since $d-1\ge3$, $F$ has precisely three simplex facets (which are ridges in $P$); let $S$ be any one of them, and denote by $G$ the other facet of $P$ containing $S$. We claim that $G$ cannot contain both $x$ and $y$. Otherwise the number of edges of $G$ containing $x$ or $y$ would only be $d$, which is absurd. Thus either every vertex in $S$ is adjacent to $x$, or every vertex in $S$ is adjacent to $y$.

It follows that one of $x,y$ is adjacent to all $d$ vertices in the two intersecting simplex facets of $F$, while the other is adjacent to all $d-1$ vertices in the other simplex facet. This enables us to describe the entire face lattice, and show that $P$ is a pentasm.

From the first paragraph, we know that every facet of $P$, besides $F$, is the ``other facet" corresponding to a ridge contained in $F$. For any of the three simplex $(d-2)$-faces contained in $F$, the other facet will contain precisely one of $x,y$, and so will be another simplex. For a prism ridge, the other facet must contain both $x$ and $y$, with each adjacent to all and only the vertices in one simplex subridge of the prism. Likewise, for a pentasm ridge, the other facet must contain both $x$ and $y$. This completely describes the vertex-facet relationships of $P$, and they correspond to those of a pentasm, as enumerated at the beginning of this section.
\end{proof}

We will see shortly that the pentasm is the unique minimiser of the number of edges, for $d$-polytopes with $2d+1$ vertices, provided $d\ge5$.  For smaller $d$,  we can exhibit now two other minimisers which are sums of triangles.

For $m,n>0$, the polytope $\Delta_{m,n}$   is be defined as the sum of an $m$-dimensional simplex and an $n$-dimensional simplex, lying in complementary subspaces. It is easy to see that it has dimension $m+n$, $(m+1)(n+1)$ vertices, $m+n+2$ facets, and is simple.  Moreover  $\Delta_{d-1,1}$ is combinatorially equivalent to the prism $M_{d,0}$. For now, we are only interested in $\Delta_{2,2}$, because it has the same number of vertices but fewer edges than the 4-dimensional pentasm. It is illustrated in \cref{fig:sumsoftriangles} (a); the labels on the vertices are needed for the following proof.

The other example, illustrated in \cref{fig:sumsoftriangles} (b), is a certain hexahedron which can be expressed as the sum of two
triangles. We will call it $\Sigma_3$; one concrete realisation of it is given by the convex hull of
$\{0, e_1, e_2, e_1+e_2,  e_1+e_3, e_2+e_3, e_1+e_2+2e_3\}$. This is the first in a sequence of $d$-polytopes $\Sigma_d$ which each can be expressed as the sum of two $(d-1)$-dimensional simplices. The higher dimensional versions have $3d-2$ vertices, only one of which is not simple; these will studied elsewhere \cite{PUY}.
Gr\"unbaum also used this as an example; a Schlegel diagram of it appears as \cite[Figure  10.4.2]{G}.

\begin{figure}[h]
\centering
\includegraphics{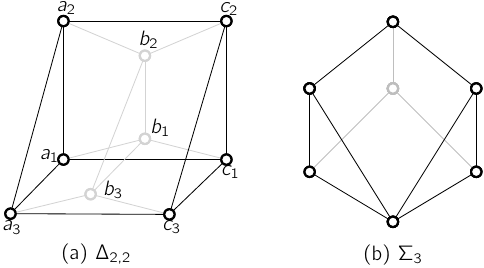}
\caption{Sums of triangles}
\label{fig:sumsoftriangles}
\end{figure}

\begin{lemma}\label{lowdim}
(i) $\Sigma_3$ is not a facet of any 4-polytope with 19 edges.

(ii) $\Delta_{2,2}$ is not a facet of any 5-polytope with 29 edges.

\end{lemma}

\begin{proof}
(i) Let us consider the possibility that $P$ is such a polytope, i.e. it has a facet $F$ of the type $\Sigma_{3}$.

A pyramid over $\Sigma_3$ would have 18 edges, so $P$ is not such a pyramid over $F$. Any 4-polytope with 10 or more vertices has more than 20 edges. So there are exactly two vertices of $P$ outside $F$, which we denote by $x$ and $y$; they must be adjacent. (It is well known that removal of an entire face from the graph of a polytope does not disconnect it \cite[Theorem 15.5]{Br}.) Since 19=11+7+1,  every vertex in $F$ belongs to only one edge not in $F$, i.e. is adjacent to exactly one of $x,y$.

Given a triangular ridge in $F$, what is the other facet containing it? It must have either four or five vertices, and each vertex in the ridge must have degree three in this facet; a simplex is the only possibility. This implies that the five vertices in the two triangles in $F$ are all adjacent to the same external vertex, say to $x$. But then $y$ can be adjacent only to the other two vertices in $F$, and so will have degree only three.

(ii) Let us consider the possibility that $P$ is such a polytope, with $\Delta_{2,2}$ as a facet, say $F$. We may label the vertices of $F$ as $a_1, a_2, a_3, b_1, b_2, b_3, c_1, c_2, c_3$ in such a way that two vertices are adjacent if and only if they share a number or a letter; see \cref{fig:sumsoftriangles} (a).
Then the facets of $F$ (which are ridges in $P$) are the convex hulls of the six subsets which either omit one number or omit one letter. Accordingly we will denote them as $R_{12}$, $R_{13}$, $R_{23}$, $R_{ab}$, $R_{ac}$ and $R_{bc}$.

 Again $P$ cannot be a pyramid over $F$, and any 5-polytope with 12 or more vertices has at least 30 edges. So there are exactly two (necessarily adjacent) vertices of $P$ outside $F$, which we denote by $x$ and $y$. If two of the vertices in $F$ are not simple in $P$, there will be at least 11 edges between $F$ and $x,y$, hence at least 30 edges in $P$.

 Thus all but at most one of the vertices in $F$ are simple in $P$; we may assume that all vertices of $F$ except perhaps $c_3$ are simple in $P$. This means that all vertices in $R_{ab}$  and $R_{12}$ are simple in $P$. Suppose the ``other facet" containing one of these ridges is a pyramid thereover, say the convex hull of $R_{ab}$ and $x$. Then none of the vertices $R_{ab}$  can be adjacent to $y$, leading to the absurdity that $y$ has degree at most 4.

This leaves us with the case when the ``other facets" for  $R_{ab}$  and  $R_{12}$ both contain both $x$ and $y$.  \cref{decomp} ensures that both facets are 4-prisms. Considering $R_{ab}$, we can suppose that $x$ is adjacent to $a_1, a_2, a_3$ and $y$ is adjacent to $b_1, b_2, b_3$. Considering likewise $R_{12}$, one of $x,y$ must be adjacent to $a_1, b_1, c_1$ while the other is adjacent to $a_2, b_2, c_2$. These conditions are  incompatible.\end{proof}

\begin{lemma}\label{3outside} Let $P$ be a $d$-polytope with $2d+1$ vertices and no more than $d^2+d-1$ edges, and suppose $P$ contains a facet $F$ with exactly $2d-2$ vertices. Then $F$ is a prism.  If $d\ge5$, then $P$ also contains a facet with $2d-1$ vertices.
\end{lemma}

\begin{proof} A 2-face with four vertices is  necessarily a prism, so assume now $d>3$.
By \cref{outside}, the three vertices outside $F$ must belong to at least $3d-3$ edges, so there are at most $d^2-2d+2=(d-1)^2+1$ edges in $F$.  Considering the degrees of its vertices, we see that $F$ has at least $(d-1)^2$ edges, and so there are either $3d-3$ or $3d-2$ edges outside $F$.

If there are  $3d-2$ edges outside $F$, then $F$ can only have  $(d-1)^2$ edges, and so must be a prism by \cref{decomp}.

If there are  $3d-3$ edges outside $F$, then $F$ has at most  $(d-1)^2+1$ edges, and $2(d-1)$ vertices. According to \cite[Theorem 13]{PYmore}, either $F$ is a prism, or $d-1=3$. In the latter case, $F$ has six vertices and there are six edges running out of $F$; by  \cref{shp}, $P$ must be decomposable. A special case of \cite[Theorem 2]{Mc} (in which the family has just one element) asserts that a polytope must be indecomposable if it has an indecomposable facet which intersects every other facet. Since no facet is disjoint from $F$, this result ensures that $F$ is also decomposable. Thus $F$ is a prism in this case as well.

We claim now that, whenever two of the three vertices outside $F$ are adjacent, one of them is adjacent to only $d-2$ vertices in $F$.
In case there are $3d-3$ edges outside $F$ this is clear, as each of the three vertices outside $F$ must  be adjacent to the other two, and to exactly $d-2$ vertices in $F$. If there are  $3d-2$ edges outside $F$, then either the three vertices outside $F$ are adjacent to one another, two of them are adjacent to $d-2$ vertices in $F$, and the third is adjacent to  $d-1$ vertices in $F$; or  two of them are each adjacent to $d-1$ vertices in $F$ but not to each other, and  the third is adjacent to both of them and to $d-2$ vertices in $F$. The claim is also clear in either of these cases.

There are $d+1$ ridges of $P$ contained in $F$, of which $d-1$ are prisms. Choose one such ridge $R$.  Then $R$ is the convex hull of two simplices, each containing $d-2$ vertices. Now consider the other facet containing $R$. If it were a pyramid over $R$, the union of these two facets would contain $d^2-3$ edges, while the remaining two vertices must be incident to at least $2d-1$ edges. Since $2d-4>d-1$, this cannot be. If it contains two of the three vertices outside $F$, it must  also be a prism, with each of these two vertices connected to all $d-2$ vertices in one of the simplices just mentioned. But each such pair of vertices outside $F$ can, according to the claim in the preceding paragraph, only be associated in this manner to one such ridge. Since there are at most three such pairs, and $F$ contains  $d-1$ such ridges,  the assumption $d\ge5$ will imply that there is a ridge whose ``other facet"  contains all three vertices outside $F$, i.e. this facet has $2d-1$ vertices altogether.
\end{proof}

This lemma raises the next question; we will see it again in  \cref{prism}.

\begin{question}
What can be said about polytopes in which every facet is either a prism or a simplex? \end{question}

\begin{theorem}\label{pentam} The $d$-polytopes with $2d+1$ vertices and $d^2+d-1$ or fewer edges are as follows.

(i) For $d=3$, there are exactly two polyhedra with 7 vertices and 11 edges; the
pentasm, and $\Sigma_3$. None have fewer edges.

(ii) For $d=4$, a sum of two
triangles $\Delta_{2,2}$ is the unique polytope with 18 edges, and the pentasm is the unique polytope with 19 edges.  None have fewer edges.

(iii) For $d\ge5$,  the pentasm is the unique $d$-polytope with $d^2+d-1$ edges. None have fewer edges.

\end{theorem}

\begin{proof}  We begin with some observations which are valid in all dimensions. First $P$ cannot be a pyramid. If it were, with some facet $F$ as its base, then $F$ would have $2d$ vertices and hence at least $\h(d-1)2d$ edges. Adding these up, $P$ would have at least $d^2+d$ edges, contrary to hypothesis. So there are at least two vertices of $P$ outside any facet.

Secondly, it is  not possible that every facet of $F$ has $d+1$ or fewer vertices; indeed any such polytope will have at least $d^2+2d-3$ edges. This follows from  \cref{excess}(ii) if every 2-face is a triangle. The only $(d-1)$-polytope with $d+1$ (or fewer) vertices and  a non-triangular face is $M_{2,d-3}$; suppose this is a facet. It contains  $M_{2,d-4}$ as a ridge, and the other facet containing this ridge can only be another copy of $M_{2,d-3}$. The union of these two facets contains ${d\choose2}-2+2d$ edges, and  \cref{outside} ensures that the $d-1$ vertices outside these 2 facets belong to at least $d(d-1)-{d-1\choose2}$ edges: summing these gives us $d^2+2d-3$.

In particular, we see that $P$ is not simplicial.

(i) Since every vertex must have degree at least three, a polytope with 7 vertices must have at least 11 edges. Suppose it has exactly 11 edges. A hexagonal pyramid has 12 edges, so every face must have at most five vertices. If some face is a pentagon, the equality 11=5+5+1 ensures that each vertex of the pentagon is adjacent to exactly one of the two other vertices, of which one is simple and one has degree four. The resulting graph is that of a pentasm. A simplicial polyhedron would have too many edges, so the remaining case is that one face is quadrilateral and that there are no pentagons.    There are two  possibilities to consider: either each vertex of the quadrilateral is adjacent to exactly one of the three other vertices, which are all adjacent to one another; or one vertex of the quadrilateral is adjacent to two of the three other vertices, which are both adjacent to the third vertex but not to each other. These two graphs are isomorphic, and coincide with the graph of $\Sigma_3$.

This can also be verified from examination of catalogues \cite{F,BD}.

(ii) It is known and easily checked that $\Delta_{2,2}$ is a simple polychoron with 9 vertices and 18 edges, and that a pentasm has  9 vertices and 19 edges.

Conversely, suppose $P$ is a 4-dimensional polytope  with 9 vertices and 18  or 19 edges. By our earlier remarks, there must be a facet with six or seven vertices.

Consider the possibility that no facet has seven vertices. \cref{3outside} then implies that at least one facet, say $F$, is a prism. If $Q$ is a quadrilateral ridge in $F$, the other facet $F'$ containing it cannot be a pyramid, as \cref{outside} would then force $P$ to have at least 20 edges. Thus $F'$ must also be a prism. Considering the three quadrilateral ridges in $F$, we see that there are three other facets of $P$ which are prisms. So, there is only one way the three vertices outside $F$ can be connected to form the graph of a polytope, and it is the graph of  $\Delta_{2,2}$. By simplicity, the entire face lattice is determined \cite[\S3.4]{Z}.

If some facet $F$ has seven vertices, the equation $19= 11+7+1$ tells us that $F$ has 11 edges, and that every vertex in $F$ belongs to only one edge not in $F$. Part (i) and  \cref{lowdim}(i) ensure that $F$ is a pentasm, and then  \cref{pentasm} ensures that $P$ is also a pentasm.

(iii) Now we proceed by induction on $d$, building on the case $d=4$.

Consider  first the possibility that there are between 4 and $d-1$ vertices outside some facet $F$. Then for some $k$ with {$3\le k\le d-2$}, we can say that $F$ has
$d-1+k$ vertices, and there are $d-k+2$ vertices outside $F$. By previous results, the number of edges in $F$ is then at least
$$\phi(d-1+k,d-1)$$
and there are at least
$$d(d-k+2)-{d-k+2 \choose 2}$$
edges outside $F$. Adding these up, the total number of edges is at least
$$d^2-1+dk+k-k^2=(d^2+d-1)+(k-2)(d-k-1)-1.$$
The positive integers $k-2$ and $d-k-1$ cannot both be equal to 1,  unless $d=5$ and $k=3$. Thus, with this exception, the number of edges is strictly more than $d^2+d-1$.

 Let us consider the case  $d=5$ and $k=3$, i.e. there is a facet with seven vertices. The four vertices outside belong to at least 14 edges, so there can only be 15 edges in the facet, which must be $M_{3,1}$ according to \cref{big}. This contains the prism $M_{3,0}$ as a ridge, and the
 other facet containing it has at least seven vertices. If it has exactly seven, the
only possibility which we need to consider for this other facet is  $M_{3,1}$. But then the union of these two facets contains 21 edges, and the three other vertices belong to at least 12 edges.  If the other facet has eight vertices, similar arithmetic shows that $P$ has at least 31 edges. If the other facet has nine vertices, then it must have at least 19 edges by  \cref{lowdim}(ii) and part (ii) above, forcing $P$ to have at least 30 edges. The existence of a facet with 10 vertices would also mean $P$ having at least 30 edges.

So  in all cases, no facet can have between $d+2$ and $2d-3$ vertices. The case that every facet has $d+1$ or fewer vertices has already been excluded, as has the case that $P$ is a pyramid.

So there is a facet with either $2d-2$ or $2d-1$ vertices. But the former possibility implies the latter, by  \cref{3outside}.

Now we can fix a facet $F$ with  exactly $2d-1=2(d-1)+1$ vertices; then  there are exactly two vertices of $P$ outside $F$. There must be at least $2d-1$ edges running out of $F$, and there must be an edge between the two external vertices. But then the total number of edges  in $F$ is at most
$d^2-d-1=(d-1)^2+(d-1)-1$. By induction, $F$ is a pentasm. (In case $d=5$, we must also apply  \cref{lowdim}(ii) again.) By  \cref{pentasm}, so is $P$.

\end{proof}

Summing up, we now see that $\min E(2d+1,d)=d^2+d-1$ for all $d\ne4$, and $\min E(9,4)=18$.

We  can show that $\min E(2d+2,d)=(d+3)(d-1)$ for all $d\ge6$; this was already known for $d=3$ or 4. This lies somewhat deeper than the results presented here; details will appear  elsewhere. The existence of a simple 5-polytope with 12 vertices, namely $\Delta_{2,3}$, is well known, so $\min E(12,5)=30$. Likewise, the existence of $\Delta_{2,4}$ shows that $\min E(15,6)=45=\min E(14,6)$.

\section{Stronger minimisation}

We have now evaluated $\min E(v,d)$ for $v\le2d+1$. It is well known that $\max E(v,d)={v\choose2}$ for all $v$ whenever $d\ge4$ (cyclic polytopes), and Steinitz showed that $\max E(v,3)=3d-6$ (again, for all $v$). Not content with knowing just its maximum and minimum, we now seek more information about the structure of $E(v,d)$. Is it a complete interval (of integers) or are some values missing?

When $v=d+1$, the only $d$-polytope is the simplex, so $E(d+1,d)=\{{d+1\choose2}\}$.
For $v=d+2$, Gr\"unbaum \cite[\S6.1]{G} described the situation   in detail. As ${v\choose2}-\min E(d+2,d)=2{v-d\choose2}=2$ when $d\ge4$, only three  values for $e$ are possible, and there exist $d$-polytopes exemplifying each of them. So $E(d+2,d)$ is a complete interval. (The corresponding  excess degrees are $d-2,d$ and $d+2$.)

For $d=3$, it is a well known consequence of Steinitz' work \cite[\S10.3]{G} that $E(6,3)=[9,12]$. For $d\ge4$, it is easy to check that $E(d+3,d)$ is also a complete interval. (The seven values in this interval correspond to the (even) excess degrees from $2d-6$ to $2d+6$.)

\begin{proposition} If $d\ge4$ and $x$ is an integer between $d-3$ and $d+3$, then there is a $d$-dimensional polytope with $d+3$ vertices and $x+\h d(d+3)$ edges.
\end{proposition}

\begin{proof}
By induction on $d$. The base case $d=4$ is due to Gr\"unbaum \cite[\S10.4]{G}. The inductive step is easy; a pyramid over an example with dimension one less always works.
\end{proof}

For the case of $d+4$ vertices, this straightforward situation no longer holds. {Gr\"unbaum \cite[Sec.~10.4]{G}} first noticed this, proving that $E(8,4)=\{16\}\cup[18,28]$. We show that the non-existence of a $4$-polytope with 8 vertices and 17 edges is not an isolated curiosity, but the beginning of a family  of natural gaps in the collection of $f$-vectors. In particular, we show  that a $d$-polytope with $v = d + 4$ or $d+5$ vertices cannot have $\phi(v, d) + 1$ edges. In other words, if our polytope is not a triplex, then it has at least two more edges than the triplex. For $d+6$ or more vertices, stronger conclusions hold.

\begin{proposition}\label{prism} Fix $d\ge4$, and let $P$ be a non-simplicial $d$-polytope with $2d$ vertices in which every facet is either a simplicial prism or a simplex. Then $P$ is a simplicial prism.
\end{proposition}

{\bf Remark.} It is easy to see that the corresponding result for $d=3$ is false.

\begin{proof} First we establish the more interesting case $d\ge5$. Suppose that $u_1,\ldots,u_{d-1}$, $v_1,\ldots,v_{d-1}$ are the vertices of a prism facet, with the natural adjacency relations suggested by the notation. Then $u_2,\ldots,u_{d-1},v_2,\ldots,v_{d-1}$ are the vertices of a prism ridge, whose other facet must be also a prism. Since $d-2\ge3$, we must have one of two extra vertices adjacent to all of  $u_2,\ldots,u_{d-1}$ and the other adjacent to all of  $v_2,\ldots,v_{d-1}$. We can call these two extra vertices  $u_{d}$ and $v_{d}$.

Repeating this argument with the other ridges, we see that the graph of $P$ is that of a prism and we can check the entire face lattice is that of a prism.

Now consider the case  $d=4$, and label the vertices (without using  subscripts) as $a,b,\ldots,h$. Suppose that $abcd$ is a quadrangle ridge, the intersection of two prisms, one containing $e$ and $f$, and the other containing $g$ and $h$. Without loss of generality,  $e$ is adjacent to $a,d$ and $f$ is adjacent to $b,c$. If $g$ is adjacent to $a,d$ and $h$ is adjacent to $b,c$, the proof proceeds as before.

But now, we need to consider also the  possibility that  $g$ is adjacent to $a,b$ and $h$ is adjacent to $c,d$. We will show that this case does not arise.

In this case, $cdef$ will be a ridge in the prism $abcdef$; the other facet containing it must contain $g,h$. Since $h$ is adjacent to $c,d$, we must have $g$ adjacent to $e,f$.

Likewise $bcgh$ is a ridge in the prism $abcdgh$; the other facet containing it must contain $e,f$. Since $f$ is adjacent to $b,c$, we must have $f$ not adjacent to $g$, contradicting the previous paragraph. So the presumed configuration is impossible.
\end{proof}

\begin{proposition}\label{prismgap} Fix $d\ge4$, and let $P$ be a $d$-polytope with $2d$ vertices and no more than $d^2+d-4$ edges. Then $P$ is a simplicial prism.
\end{proposition}

\begin{proof}  We use induction on $d$; the base case $d=4$ is clear from \cref{decomp}.

We will proceed by showing that every facet is a prism or a simplex, and then apply  \cref{prism}. So let $F$ be any facet of $P$
and denote by $n$ the number of vertices outside $F$; clearly $n\le d$.

Were $n$ equal to one, $P$ would be a pyramid and its  base $F$ would have $2d-1=2(d-1)+1$ vertices and at most $(d^2+d-4)-(2d-1)$ edges. But this number is
$$=(d-1)^2+(d-1)-3<{\min E(2(d-1)+1,d-1) }.$$ \cref{pentam} eliminates this possibility.

 So $n\ge2$, and $F$ has $2d-n\le2(d-1)$ vertices and hence $P$ has at least
$$\phi(2d-n,d-1)+nd-{n\choose2}=\phi(2d,d)+(d-n)(n-2)$$
edges. This number is at least $d^2+d-3$, for all $n$ between 3 and $d-1$.

If $n=d$, $F$ is obviously a simplex. This leaves us with the case $n=2$.

If every vertex in $F$ has a unique edge leading out of it, then $P$ will be decomposable by  \cref{decomp}, and hence a prism and we are finished. Otherwise, there will be at least $2d-2+1$ edges going out of $F$ and one edge between the two vertices outside $F$. This implies that there are at most
$$(d^2+d-4)-2d=(d-1)^2+(d-1)-4$$
edges in $F$. By induction, $F$ must be a prism.
This completes the proof.
\end{proof}

Rewording, any $d$-polytope with $2d$ vertices which is not a prism must have at least $d^2+d-3$ edges. This is almost best possible, since a pyramid based on a pentasm has $2d$ vertices and $d^2+d-2$ edges.

We conjecture that for $d\ge6$, there are no $d$-polytopes with $2d$ vertices and $d^2+d-3$ edges. For $d=5$ or 3, examples are easy to find.

Next we extend this result to $d+k$ vertices, for $k<d$. That is, we show that any such $d$-polytope which is not a triplex has at least $k-3$ more edges than the triplex, i.e. excess  degree at least  $(k-1)(d-k)+2(k-3)$. First we  establish a special case.

\begin{lemma}\label{tenandahalf} As usual, let $P$ be a $d$-dimensional polytope whose vertex set $V$ has $d+k$ elements, $k\le d$. Suppose that $P$ is not simplicial, that every non-simplex facet has $d+k-2$ vertices, and that every non-simplex ridge has $d+k-4$ vertices. Then either $k=3$ or $k=d$.
\end{lemma}

\begin{proof}  The hypotheses exclude the possibilities that $k=1$ or 2.
So assume $k\ge4$.

The ridge hypothesis implies that if $F$ and $G$ are distinct non-simplex facets, then the ``outside pairs" $V\setminus F$ and $V\setminus G$ will be disjoint.

If every 2-face is a triangle,  Kalai's rigidity theorem,  \cref{excess}(ii), tells us that  the excess of $P$ is at least $(k-1)d$, which is clearly more than $(k-1)(d-k)+2(k-3)$.

Otherwise, there is a non-triangular 2-face $Q$, which must belong to at least $d-2$ distinct facets, none of which can be simplicial. The ``outside pairs" of these facets total $2d-4$ vertices and there are at least 4 vertices in $Q$. So $P$ has $2d$ vertices and $k=d$.
\end{proof}

\begin{theorem}\label{biggap} If $4\le k\le d$, then a
 $d$-polytope with $v=d+k$ vertices which is not a triplex must have at least $\phi(v,d)+k-3$ edges.
In other words, its excess degree is at least $(k-1)(d-k)+2(k-3)$.
\end{theorem}

\begin{proof}  \cref{prismgap} establishes this for $k=d$, so we assume $k<d$. {We proceed by induction on $d$ for a fixed $k$.}

As before, choose a facet $F$ and let $n$ be the number of vertices of $P$ not in $F$.

If $n=1$, then $F$ is not a triplex, since $P$ is a pyramid over $F$. The inductive hypothesis ensures that $\xi(F)\ge(k-1)(d-1-k)+2(k-3)$. The apex of the pyramid has excess degree $k-1$, so adding these up gives the desired estimate. (This is the only instance in which the inductive hypothesis is needed).

For $3\le n<k$ proceed as in  \cref{big}, with the required estimate actually following from the arguments presented there.

We are left with the cases $n=2$ or $n=k$. But if these are the only possible values of $n$, then $P$ has the property that every non-simplex facet has precisely two vertices outside it. In particular, all non-simplicial facets have the same number of vertices.
Moreover, a non-simplex ridge must omit at most two vertices of any facet which contains it; otherwise the other facet containing it would omit at least three vertices without being a simplex.
The case when $P$ is simplicial  follows from the Lower Bound Theorem, so we assume that at least one facet is not a simplex.

First consider the possibility that some such facet  $F$ is a pyramid over some ridge $R$. Then the other facet $G$ containing $R$ must also be a pyramid thereover, and there will be one vertex not in $F\cup G$. Then $R$ will contain $d+k-3=d-2+k-1$ vertices, so must have excess degree at least $(k-2)(d-k-1)$ by \cref{big}. {The apices of the pyramids will each contribute excess degree $k-3$ in $P$, before we consider the contribution of the remaining vertex; call it $v$. Since $v$ belongs to at least $d$ edges, the vertices at the other end of these edges will each  contribute 1 to the excess degree of $P$}. So $P$ will have excess degree at least $(k-2)(d-k-1)+2(k-3)+d=(k-1)(d-k)+2k-4$, i.e. $P$ will have at least $k-2$ more edges than the corresponding triplex.

The remaining situation is that every non-simplex ridge has exactly two vertices less than the facets containing it. Together \cref{prismgap} and \cref{tenandahalf}  deal with this situation.
\end{proof}

The case $k=4$ in preceding theorem does not tell us anything more than  \cref{big}. We now give the promised result that $E(d+4,d)$ also contains a gap.

\begin{theorem} If $d\ge4$, then a
 $d$-polytope $P$ with $d+4$ vertices cannot have $\phi(d+4,d)+1$ edges.
In other words, either $P$ is a triplex, or its excess degree is at least $3d-8$.
\end{theorem}

\begin{proof} Again, by induction on $d$. Gr\"unbaum {\cite[Thm.~10.4.2]{G}} established the base case, $d=4$.  We note a shorter proof of this, using the structure results for polytopes with low excess. Combining \cite[Thms. 4.1 and 4.10]{PUY} shows that a polytope with excess $d-2$ must be either decomposable or a pyramid. A 4-polytope with 8 vertices and 17 edges would have excess two. However it could not be decomposable because of \cref{decomp}, and it could not be a pyramid because its base would  need to have seven vertices and only ten edges. We now proceed to the inductive step.

If some facet has $d+3$ vertices, then $P$ is a pyramid, and the conclusion follows easily by induction.

If some ridge has $d+2$ vertices, then some facet has $d+3$ vertices, and we are finished.

Now suppose there is a ridge $R$ with $d+1=d-2+3$ vertices. We only need to consider the case that both facets containing it have $d+2$ vertices. This ridge has excess degree at least $(3-1)(d-2-3)=2d-10$. Both facets are pyramids over $R$, and their apices contribute excess degree 2 in $P$. Next we consider the remaining vertex, outside both facets. It has degree at least $d$, so  is adjacent to at least $d-2$ vertices in $R$. The vertices at the other end of these edges will each  contribute 1 to the excess degree of $P$,  so we have another contribution to the excess of at least $d-2$. The total excess degree of $P$ is then at least $2d-10+2+2+d-2=3d-8$.

Henceforth we may assume that every ridge has either $d$ or $d-1$ vertices, and that every facet has at most $d+2$ vertices. Note that if a triplex $M_{3,d-4}$ is a facet of $P$, then $P$ will contain $M_{3,d-5}$ as a ridge with $d+1$ vertices. So we may also assume that any facet with $d+2=d-1+3$ vertices is not a triplex, and thus has excess degree at least $(3-1)(d-1-3)+2=2d-6$ (by virtue of \cref{excess}(iii)).

Let $R$ be a ridge with $d=d-2+2$ vertices. We assume first that it is not a triplex, and will show that  either $\xi(P)\ge3d-8$, or that there is another ridge with  $d$ vertices which is  a triplex. Not being a triplex, the excess degree of $R$ will be at least $d-2$.

If one facet containing $R$ has $d+1$ vertices, it will be a pyramid over $R$. The edges incident with the three vertices outside the facet will contribute excess degree at least $3(d-2)-d$. The total excess will then be at least $d-2+2d-6=3d-8$.

Otherwise, both facets containing $R$ have $d+2$ vertices. Having degree at least $d-1$ within the facet, each external vertex is adjacent to at least $d-2$ vertices in $R$. Hence the edges incident with the two extra vertices in each facet contribute excess degree at least $2(d-2)-d=d-4$, and the total excess degree of $P$ is  at least $d-2+d-4+d-4=3d-10$. The excess degree can be strictly less than $3d-8$ only if $R$ has excess degree exactly $d-2$ and each vertex outside $R$ is simple.

Let $F$ be one of the facets, and denote by $a,b$ the two vertices in $F\setminus R$. Then $R$ must contain two vertices (say $v_1$ and $v_2$) which are adjacent to $a$ but not to $b$, two vertices (say $w_1$ and $w_2$) which are adjacent to $b$ but not to $a$, and $d-4$ vertices which are adjacent to both $a$ and $b$. Now the graph of $R$ is almost complete, i.e. has only one edge missing. Without loss of generality, we can assume $v_1$ is adjacent to $w_1$. Let $S$ be a facet of $R$ (i.e. a ridge of $F$) containing $v_1$ and $w_1$. Denote by $R'$ the other facet of $F$ containing $S$. Clearly $v_1$ and $w_1$ must be adjacent to some vertices in $F\setminus R$; thus $R'$ contains both $a$ and $b$. Of course $R'$ is a ridge in $P$. By previous considerations, $R'$ cannot have $d+1$ vertices. Neither $aw_1$ nor $bv_1$ are edges, so $R'$ is not a simplex. Thus it has $d=d-2+2$ vertices. With two edges missing, it must be a triplex $M_{2,d-4}$.

So we consider the case that some ridge $R$ is a triplex, with $d$ vertices and excess degree $d-4$. The two facets containing it, say $F$ and $G$, may have either $d+1$ or $d+2$ vertices.

If both have $d+1$ vertices, they will be pyramids and every vertex in $F\cup G$ will have degree $d$ in the graph of $F\cup G$. Each edge between one of the two vertices outside  $F\cup G$  and a vertex in  $F\cup G$ will increase the excess degree by one, and there are at least $2(d-1)$ such edges. Consequently the total excess degree is  at least $d-4+2d-2=3d-6$.

If both such facets have $d+2$ vertices, they will have excess degree at least $2d-6$ (as they are not triplices). Summing over all the vertices, the total excess degree of $P$ will be at least $\xi(F)+\xi(G)-\xi(R)\ge2(2d-6)-(d-4)=3d-8$.

Otherwise, we can suppose that $F$ has $d+2$ vertices, and hence excess degree at least $2d-6$, while facet  $G$ has $d+1$ vertices and hence is a pyramid over $R$. Every vertex in $G$ has degree at least $d$ in $G\cup F$, and  the one vertex outside  must be adjacent to at least $d-2$ of them. Hence this vertex contributes excess degree at least $d-2$. The excess degree of $P$ is then at least $2d-6+d-2=3d-8$. This completes the proof in the case when some ridge has $d$ vertices.

Finally, we have the situation when every ridge has $d-1$ vertices, i.e. is a simplex. Rather than going through another  case by case analysis of the cardinality of the facets,  we complete the proof by appealing to  Kalai's Rigidity Theorem  (\cref{excess}(ii))  again. Every 2-face of $P$ is a triangle in this case, so the excess degree is at least that guaranteed for simplicial polytopes by the lower bound theorem i.e. $3d$.
\end{proof}

Now we can present a second result about gaps in the possible number of edges.

\begin{proposition} (i)  Fix $n\ge4$. For any $d\ge n^2$, there is  no $d$-polytope with between $\phi(d+n+1,d)+1$ and $\phi(d+n+1,d)+n-3$ edges.

(ii) If  $d=n^2-j$, where  $1\le j\le n-4$, then there is  no $d$-polytope with between $\phi(d+n+1,d)+j+1$ and $\phi(d+n+1,d)+n-3$ edges.

\end{proposition}

\begin{proof}  It is easy to check that $n<\h d$ in both cases. We  will use again  the  identity

$$\phi(d+n+1,d)={d+n\choose2}+d-n^2.$$

The two parts together are equivalent to the statement

if $d\ge n^2-j$
and  $0\le j\le n-4$, then there is  no $d$-polytope with between $\phi(d+n+1,d)+j+1$ and $\phi(d+n+1,d)+n-3$ edges.

So let $P$ be a $d$-polytope with $v$ vertices  and $e$ edges.

If $v\le d+n$, then $e\le{d+n\choose2}\le\phi(d+n+1,d)+j$.

If $v=d+n+1$ and $P$ is a triplex, then $e=\phi(d+n+1,d)$.

If $v=d+n+1$ and $P$ is not a triplex, then  \cref{biggap} ensures that $e\ge\phi(d+n+1,d)+n-2$.

If $2d\ge v\ge d+n+2$, then $$e\ge\phi(d+n+2,d)=\phi(d+n+1,d)+d-n-1>\phi(d+n+1,d)+n-2.$$

If $v>2d$, then $e\ge\h dv> d^2=\phi(2d,d)>\phi(d+n+2,d)$.\end{proof}

\section{Higher dimensional faces}

 Recall that the number of $m$-dimensional faces of a polytope $P$ is denoted by $f_m(P)$, or simply $f_m$ if $P$ is clear from the context. We will continue the study of lower bounds for high dimensional faces in this section. In particular, \S3 showed that  if $P$ is a triplex with $d+k$ vertices, then
$f_m(P)=\phi_m(d+k,d)$.

Let us define $F_m(v,d)=\{n:$ there is a $d$-polytope with $v$ vertices and $n$ faces of dimension $m\}$. Of course $F_1=E$.
As we said at the beginning, Gr\"unbaum \cite[p 184]{G}  conjectured that $\min F_m(v,d)=\phi_m(v,d)$ for $d<v\le2d$.
He proved that this is true for every $m$ and $v\le d+4$.

McMullen \cite{M} established this for the case $m=d-1$ and all $v\le2d$ (and also solved the problem of minimising facets for some $v>2d$). As far as we are aware, this is the only paper which considers any aspect of the lower bound problem for general polytopes. When $m=d-1$ and $v\le2d$, it is easy to check that $\phi_m(v,d)=d+2$. We will first show that Gr\"unbaum's conjecture is correct for $d$-polytopes with $d+2$ facets and no more than $2d$ vertices, for any value of $m$.

Using this, we will then confirm Gr\"unbaum's conjecture for $m\ge0.62d$ and $v\le2d$, also proving the triplex is the unique minimiser if in addition $m\ne d-1$. We also  present some results  concerning high dimensional faces when $v=2d+1$.

To continue, it will be necessary to understand the structure of $d$-polytopes with $d+2$ facets. The structure of $d$-polytopes with $d+2$ vertices is quite well known, \cite[\S6.1]{G} or  \cite[\S3.3]{MS}, and dualising leads to the following result, which classifies the $d$-polytopes with $d+2$ facets. It appears explicitly in \cite{M}. The calculation of the $f$-vector is the dual statement to \cite[\S6.1.4]{G}.

\begin{lemma}\label{dplus2facets}
 Any $d$-dimensional polytope with $d+2$ facets is, for some $r,s$ and $t$ with $d=r+s+t$, a $t$-fold pyramid over {$\Delta_{r,s}$}. It has $(r+1)(s+1)+t$ vertices, and the number of its $m$-dimensional faces is
$$ {d+2\choose m+2} -{s+t+1\choose m+2}-{r+t+1\choose m+2} +{t+1\choose m+2}.$$
\end{lemma}

Recall that $F_{d-1}(d+1,d)=\{d+1\}$, that $\min F_{d-1}(v,d)>d+1$ if $v>d+1$, and that every triplex other than the simplex has $d+2$ facets. Thus, amongst all $d$-polytopes with $d+k$ vertices, the triplex minimises the number of facets. In general it is not the unique minimiser. But sometimes it is; it depends on the value of $k$.  The next result  reformulates the special case of a result of McMullen \cite[Theorem 2]{M}, in which only $d$-polytopes with no more than $2d$ vertices are considered.

\begin{proposition} Fix $k$ with $2\le k\le d$. Then

(i) $\min F_{d-1}(d+k,d)=\phi_{d-1}(d+k,d)=d+2$;

(ii) the minimum is attained by $M_{k,d-k}$;

(iii) the minimiser is unique, i.e. there is only one $d$-polytope with $d+k$ vertices and $d+2$ facets, if and only if $k=2$ or $k-1$ is  a prime number.
\end{proposition}

\begin{proof}
It is routine to check that $\phi_{d-1}(d+k,d)=d+2$, so  (i) and (ii) are clear.

For (iii),  \cref{dplus2facets} tells us that we need only consider a $t$-fold pyramid over {$\Delta_{r,s}$}.
Our hypothesis tells us that $d+k=(r+1)(s+1)+t$  and  $d=r+s+t$; this forces $k = rs +1$.

So if $k-1$ is prime or 1, then $\{r,s\}=\{1,k-1\}$, $t=d-k$ and the polytope is $M_{k,d-k}$.

On the other hand,  if $k-1=rs$ where $r>1,s>1$, then $r+s\le rs+1=k\le d$, and so $t=d-r-s$ is non-negative and we have a second solution for $(r,s,t)$.
\end{proof}

We now establish that Gr\"unbaum's conjecture is correct for all $d$-polytopes with $d+2$ facets.

\begin{theorem}\label{dplus2f}
 Let $P$ be a $d$-dimensional polytope with $d+2$ facets and  $v\le2d$ vertices. If $P$ is not a triplex, and $1\le m\le d-2$, then $f_m(P)>\phi_m(v,d)$.
\end{theorem}

\begin{proof}
We will repeatedly use the well known identity
$$\sum_{j=1}^p{d-j\choose n}={d\choose n+1}-{d-p\choose n+1} $$
which follows from repeated application of Pascal's identity.

Let $r,s$ and $t$ be given by  \cref{dplus2facets}. Our hypotheses imply that  $r\ge2,s\ge2$, and $rs+1=v-d\le d$.

We first claim that
$$ {d\choose m+2} -{d-r+1\choose m+2}-{d-s+1\choose m+2} +{d-r-s+1\choose m+2}+{d-rs\choose m+1}>0.$$

Using the identity above several times, we have
\begin{eqnarray*}
&  & {d\choose m+2} -{d-r+1\choose m+2}-{d-s+1\choose m+2} +{d-r-s+1\choose m+2}+{d-rs\choose m+1}\\
& = & \sum_{i=1}^{r-1}{d-i\choose m+1}-\sum_{i=1}^{r}{d-s+1-i\choose m+1}+{d-rs\choose m+1} \\
& = & \sum_{i=1}^{r-1}\bigg({d-i\choose m+1}-{d-s+1-i\choose m+1}\bigg)-{d-r-s+1\choose m+1}+{d-rs\choose m+1} \\
& = & \sum_{i=1}^{r-1}\sum_{j=1}^{s-1}{d-i-j\choose m}-{d-r-s+1\choose m+1}+{d-r-s+1-(r-1)(s-1)\choose m+1} \\
& = & \sum_{i=1}^{r-1}\sum_{j=1}^{s-1}{d-i-j\choose m}-\sum_{k=1}^{(r-1)(s-1)}{d-r-s+1-k\choose m} \\
& \ge & (r-1)(s-1){d-(r-1)-(s-1)\choose m}-(r-1)(s-1){d-r-s\choose m} \\
& = & (r-1)(s-1)\bigg({d-r-s+2)\choose m}-{d-r-s\choose m}\bigg)\ge0,
\end{eqnarray*}
as required.  To complete the proof of the claim, note that the last inequality is strict if $r=s=2$, and the previous inequality is strict otherwise.

Adding ${d+1\choose m+1}+{d\choose m+1}$ to both sides and rearranging, we obtain
$${d+2\choose m+2} -{d-r+1\choose m+2}-{d-s+1\choose m+2} +{d-r-s+1\choose m+2}>{d+1\choose m+1}+{d\choose m+1}-{d-rs\choose m+1}.$$

Recalling that $d=r+s+t$ and $2d+1-v=d-rs$, this is precisely the assertion that $f_m(P)>\phi_m(v,d)$.
\end{proof}

The  following technical result plays an important role in the next theorem.

\begin{lemma}\label{ugly}
Let $\alpha=\h(\sqrt5-1)$ denote the reciprocal of the golden ratio, and let $\beta=0.543689\ldots$ be  defined by $3\beta=(3\sqrt{33}+17)^{1/3}-(3\sqrt{33}-17)^{1/3}-1$.

(i) For all  integers $d$ and $m$ with $d\ge m\ge2$,
$${d\choose m}-{d\choose m+1}-{d-2\choose m-2}=\frac{m^2+dm-(d-1)^2}{(m+1)m}{d-2\choose m-1},$$
and this expression is strictly positive if either $m\ge\alpha d$, or if $m\ge{\frac35}(d-1)$ and $d\le15$.

(ii) For all  integers $d$ and $m$ with $d\ge m\ge3$,
$${d\choose m}-{d\choose m+1}-{d-3\choose m-3}=\frac{p(m,d)}{(m+1)m(m-1)}{d-3\choose m-2},$$
where $p(m,d)=m^3+(d-2)m^2+(d^2-2d-1)m-(d^3-4d^2+5d-2)$, and this expression is strictly positive if either $m\ge\beta d$, or if $m\ge\h d$ and $d\le17$.

\end{lemma}

\begin{proof}
(i) The proof of the combinatorial identity is  tedious but routine. If $m\ge\alpha d$, then $m^2+dm-(d-1)^2\ge2d-1$. If $m\ge{\frac35}(d-1)$, then  $m^2+dm-(d-1)^2\ge{\frac{1}{25}}(d-1)(16-d)$.

(ii) Likewise,  noting that $\beta$ is the root of the equation $x^3+x^2+x=1$.

\end{proof}

For  high dimensional faces other than  facets, the triplex is the unique minimiser.

\begin{theorem}
Fix $d, m,k$ with $k\le d$ and $m\ge0.62d$ (or $m\ge0.6(d-1)$ and $d\le15$), and let $P$ be a $d$-polytope with $d+k$ vertices. If $P$ is a triplex, then $f_m(P)=\phi_m(d+k,d)$. If $P$ is not a triplex,  and $m\ne d-1$, then $f_m(P)>\phi_m(d+k,d)$.
\end{theorem}

\begin{proof}

The conclusion about triplices was noted earlier. Henceforth, assume that $P$ is not a triplex.
Thanks to  \cref{dplus2f}, we may also suppose $P$ has $d+3$ or more facets.

Then the dual polytope $P^*$ has at least $d+3$ vertices. According to \cite[10.2.2]{G}, we then have
\begin{eqnarray*}
f_m(P)=f_{d-m-1}(P^*)&\ge&\phi_{d-m-1}(d+3,d)\\
&=&{d+1\choose d-m}+{d\choose d-m}-{d-2\choose d-m}\\
&=&{d+1\choose m+1}+{d\choose m}-{d-2\choose m-2}.
\end{eqnarray*}

Recalling the definition of $\phi_m$, and applying Pascal's identity, we then have
$$
f_m(P)-\phi_{m}(d+k,d)\ge{d\choose m}-{d\choose m+1}-{d-2\choose m-2}+
{d+1-k\choose m+1}.
$$

\cref{ugly}(i) guarantees that this is strictly positive.

\end{proof}

We are now able to confirm Gr\"unbaum's conjecture for $d\le5$.  Of course he proved it for $k\le4$, so need only consider the
case $d = k = 5$. Within this case, we have proved it now for $m=1$ and $m\ge3$, and thus we fix $m=2$.  So let $P$ be a 5-dimensional polytope with 10 vertices, $e$ edges, $t$ 2-dimensional faces, $r$ ridges and $f$ facets.  A prism has $\phi_2(10,5)=30$ 2-faces. If $P$ is not a triplex,  \cref{biggap} ensures $e\ge27$. If $f=7$, then $P$ is a pyramid over $\Delta_{2,2}$, and then $t=33$. So suppose $f\ge8$.
Euler's relation tells us that $10-e+t-r+f=2$, and the dimension ensures $5f\le2r$. It follows that
$$t=(e-8)+(r-f)\ge19+\h(3f)\ge31>\phi_2(10,5).$$

Thus the first cases for which Gr\"unbaum's conjecture remain open are $d=6$, $m=2$, $v=11$ or 12.

This proof actually shows that there are gaps in the number of $m$-faces for  values of $m$ other than 1,  something which has not been previously observed. For example, in dimension 5, the triplex $M_{3,2}$ has 20 ridges, but every other 5-polytope with 8 vertices has at least 22 ridges.

Under the additional assumption that $f_{d-1}(P)\ne d+3$, a slightly stronger conclusion is possible.

\begin{proposition}
Fix $d, m,k$ with $k\le d$ and $m\ge0.55d$ (or $m\ge0.5d$ and $d\le17$), and let $P$ be a $d$-polytope with $d+k$ vertices. If $P$ has $d+4$ or more facets, then $f_m(P)>\phi_m(d+k,d)$.
\end{proposition}

\begin{proof}
Much  as before, using Gr\"unbaum's result  \cite[10.2.2]{G} that $f_{d-m-1}(P^*)\ge\phi_{d-m-1}(d+4,d)$, and \cref{ugly}(ii).
\end{proof}

Since we have investigated the minimal number of edges for $d$-polytopes with $2d+1$ vertices, we will do the same for facets and ridges. For $k=d+1$, the last step of the previous proof breaks down. Before continuing, we rephrase the remaining case of \cite[Theorem 2]{M}.

\begin{proposition}
Fix  $k>d$ and consider the class of $d$-polytopes with $d+k$ vertices. Then this class contains a polytope with $d+2$ facets if, and only if, $k-1$ is a composite number, say $rs$, with $r+s\le d$. Different  factorisations of $k-1$ give rise to combinatorially distinct polytopes.
\end{proposition}

\begin{proof}
Again by  \cref{dplus2facets}, the existence of such a polytope $P$ is equivalent to the existence of $r\ge1,s\ge1, t\ge0$ with $d+k=(r+1)(s+1)+t$  and  $d=r+s+t$. This implies that $k-1=rs$, and we cannot have $r$ or $s=1$, because then $r+s=1+k-1>d$. Conversely, given $r$ and $s$, put $t=d-r-s$ and consider a $t$-fold pyramid over $\Delta_{r,s}$.
\end{proof}

Returning briefly to the question of monotonicity, this result shows  that $\min F_4(11,5)=8$ but $\min F_4(12,5)=7$. So for fixed $m$ and $d$, $\min F_{m}(v,d)$ is not generally a monotonic function of $v$.

If $r+s\le d$, then $rs\le\frac{1}{4} d^2$, and $P$ can have at most $\frac{1}{4}d^2+d+1$ vertices. McMullen \cite[p 352]{M} showed that, for $d+2\le v\le\frac{1}{4}d^2+2d$, there is a $d$-polytope with $v$ vertices and $d+3$ facets. In particular, when $\frac{1}{4}d^2+d+1< v\le\frac{1}{4}d^2+2d$, he
proved that  $\min F_{d-1}(v,d)=d+3$.
By  \cite[Theorem 2]{M}, for $ 2d+1\le v\le\frac{1}{4}d^2+d+1$, $\min F_{d-1}(v,d)$ is either $d+3$ or $d+2$, depending on whether $v-d-1$ is prime or composite.

But we continue to restrict our attention to polytopes with $2d+1$ or fewer vertices. For facets, the next result reformulates McMullen's work in this special case, with a different proof.

\begin{proposition}
Consider the class of $d$-polytopes with $2d+1$ vertices.

(i) If $d$ is a prime, the minimal possible number of facets is $d+3$, and the minimiser is not unique.

(ii) If $d$ is a product of 2 primes, the minimal number of facets is $d+2$, and the minimiser is unique.

(iii) If $d$ is a product of 3 or more primes, the minimal possible number of facets is $d+2$, and the minimiser is not unique.
\end{proposition}

\begin{proof}

(i) If $d$ is prime,  the previous result ensures that there is no $d$-polytope with $2d+1$ vertices and $d+2$ facets. We need to show that there at least two $d$-polytopes with $2d+1$ vertices and $d+3$ facets. \cref{pentam} tells that there are precisely 2 such examples if $d=3$.

For $d\ge4$, the structure of $d$-polytopes with $d+3$ facets is moderately well understood {\cite[\S6.2 \& \S6.7]{G} or \cite{Fu}, so the existence of two distinct such polytopes should come as no surprise. Our work so far makes it easy to give two examples; the rest of this paragraph does not require $d$ to be prime. The pentasm is one obvious example. For a second, consider the pyramid whose base is the Minkowski sum of a line segment and $M_{2,d-4}$. This triplex has dimension $d-2$, $d$ vertices and $d$ facets; its direct sum with a segment has dimension $d-1$, $2d$ vertices and $d+2$ facets; and a pyramid thereover has dimension $d$, $2d+1$ vertices and $d+3$ facets. (It is likely that two is a serious underestimate of the number of examples; in dimension four, there are six examples \cite[Figure 5]{Gru70}.)}

(ii) and (iii) follow from the preceding result. \end{proof}

Finally, we announce the corresponding result for ridges. The proof  is much longer, and will appear elsewhere.

\begin{proposition}
Consider the class of $d$-polytopes with $2d+1$ vertices.

(i) If $d$ is a prime, the minimal number of ridges is $\h(d^2+5d-2)$, and the pentasm is the unique minimiser.

(ii) If $d$ is a product of two primes, the minimal number of ridges is $\h(d^2+3d+2)$, and the minimiser is unique.

(iii) If $d$ is a product of three or more primes, the minimal number of ridges is $\h(d^2+3d+2)$, and the minimiser is not unique.
\end{proposition}

We agree with McMullen \cite[p 351]{M}; for dimensions
$d\ge6$, the problem of determining $\min F_m(v, d)$ for $2\le m\le\h d$ appears to be extremely
difficult.

\section{Acknowledgments} We thank Eran Nevo for assistance with the translation of  \cite[Theorem 7.1]{Kal}. 

The research of Guillermo Pineda-Villavicencio was supported partly by the Indonesian government Scheme P3MI, Grant No. 1016/I1.C01/PL/2017, and partly by a grant from the Capital Markets Cooperative Research Centre.
Julien Ugon is the recipient of an Australian Research Council Discovery Project (project number DP180100602) funded by the Australian Government.

\end{document}